 \documentclass[11pt,twoside,reqno,centertags,draft]{amsart}

 \usepackage{amsmath,amsthm,amsfonts,amssymb}
\usepackage[mathscr]{euscript}
 \usepackage{hyperref} 
 \usepackage{mathrsfs} 
\usepackage{graphicx}
  \usepackage{epsfig}

\newtheorem{theorem}{Theorem}[section]
\newtheorem{lemma}[theorem]{Lemma}

\newenvironment{definition}[1][Definition]{\begin{trivlist}
\item[\hskip \labelsep {\bfseries #1}]}{\end{trivlist}}

\newenvironment{remark}[1][Remark]{\begin{trivlist}
\item[\hskip \labelsep {\bfseries #1}]}{\end{trivlist}}

\newcommand{\R}{\mathbb{R}}
\newcommand{\Z}{\mathbb{Z}}

\newcommand{\eps}{\varepsilon}
\newcommand{\be}{\begin{equation}}
\newcommand{\ee}{\end{equation}}
\newcommand{\bea}{\begin{eqnarray}}
\newcommand{\eea}{\end{eqnarray}}

\newcommand{\ba}{\begin{array}}
\newcommand{\ea}{\end{array}}

\begin{document}
 
\title{Orbital stability of localized structures via B\"acklund transfomations}
\date{}
\maketitle     
 
\vspace{ -1\baselineskip}

{\small
\begin{center}
\renewcommand{\thefootnote}{\fnsymbol{footnote}}
 {\sc A. Hoffman\footnote{Franklin W. Olin College of Engineering; Needham, MA} and C.E. Wayne\footnote{Boston University Department of Mathematics and Statistics; Boston, MA} 
} 
\end{center}
}

\numberwithin{equation}{section}
\allowdisplaybreaks
\renewcommand{\thefootnote}{\arabic{footnote}}

 \begin{quote}
\footnotesize
{\bf Abstract.}  
The B\"acklund Transform, first developed in the context of differential geometry, has been classically used to obtain multi-soliton states in completely integrable infinite dimensional dynamical systems.  It has recently been used to study the stability of these special solutions.  We offer here a dynamical perspective on the B\"acklund Transform, prove an abstract orbital stability theorem, and demonstrate its utility by applying it to the sine-Gordon equation and the Toda lattice.
\end{quote}

\section{Introduction}








In this paper we survey some recent work on the use of B\"acklund transformations to study the stability of
localized structures in infinite dimensional Hamiltonian systems.  For finite dimensional Hamiltonian systems
the constraints imposed by the Hamiltonian structure mean that the stability of stationary solutions
can be reduced either to showing that all the eigenvalues of the linearized system at the fixed point
lie on the imaginary axis (for spectral stability) or that the full, nonlinear system exhibits Lyapunov stability.
The stability of periodic orbits can be studied by similar methods by reducing the problem to the consideration
of a fixed point of a Poincar\'e map.

For infinite dimensional systems the situation can be more subtle.  There, the possible presence of dispersive
phenomena means that one may have asymptotic stability of such systems, in appropriate norms, a 
phenomenon that is impossible in finite
dimensional systems.

Consequently, the study of stability in such systems has followed two rather different tracks.  On one hand,
methods to prove Lyapunov (or orbital) stability of localized solutions like traveling waves, solitons, or
multi-solitons have been developed which rely on regarding the solution as a minimizer, or critical
point, of some energy functions, often subject to appropriate constraints.  Examples of this
type of approach are \cite{benjamin:1972}, \cite{grillakis:1987}, \cite{bona:1987}, \cite{maddocks:1993}.

The second approach typically begins by analyzing the linearization of the system about the solitary wave.
The spectrum of the linearization is then considered and one shows that on the complement of the point
spectrum the linearized evolution generates a dispersive evolution.  If the dispersive
decay is sufficiently rapid, this can then (sometimes) be used in conjunction with Duhamel's formula to
derive nonlinear, asymptotic stability of the underlying solitary wave.  Examples of this approach
include \cite{pego:1994}, \cite{martel:2002}.  Closely related to this approach are stability or instability results based on invariant
manifold theorems \cite{krieger:2006}, \cite{pillet:1997}.  Here, one typically shows that the nonlinear equations possesses 
an invariant manifold associated with the family of localized solutions and examines the behavior
of solutions near this manifold to understand stability properties of the underlying family.

Recently, an old tool has been adapted to study both of these types of stability, namely B\"acklund transformations.
B\"acklund transformations define a relationship between two functions (often, through a differential equation,
or some more complicated equation) such that if one of the functions satisfies a given partial differential
equation, so does the second.  The partial differential equation satisfied by the second function may
be the same PDE satisfied by the first function (in which case one speaks of an auto-B\"acklund transform ) or it may be a different
PDE.  In the study of infinite-dimensional Hamiltonian systems, B\"acklund transforms have mostly been used in the context of
completely integrable infinite dimensional systems to obtain explicit formulas for soliton, multi-soliton, or other
special solutions of the equations.  Thus, for instance, the auto-B\"acklund transform for the Korteweg-de Vries (KdV) equation
relates the zero solution to the one-soliton solution, the one-soliton solution to the two-soliton solution and
so on and so forth.

However, the B\"acklund transformation is also turning out to be a useful tool to investigate the stability of such special solutions
as well.  It may not be clear at first glance why this is of interest.  Since in principle, one knows ``everything'' about
solutions of a completely integrable system, the stability or instability of such solutions might seem an obvious
by-product of their integrability.  In practice, however, it may be difficult to see from the formulas defining
the solutions in these integrable systems what the asymptotic behavior of solutions with initial conditions
close to a soliton are.  Furthermore, stability results based on B\"acklund transformations have yielded at least two new insights
not available from the complete integrability machinery:
\begin{itemize}
\item First, in some circumstances, they allow one to establish stability in much less regular spaces than can
be treated either with completely integrable structure, or with energy methods.  A first example of this approach is
the work of Merle and Vega, \cite{merle:2003}.  They used the Gardner transformation which maps solutions of the KdV
equation close to the soliton into solutions of the modified KdV equation near a kink solution.  (The Gardner
transformation is an example of a B\"acklund transformation which links two different equations.)  They then use the stability
of modified KdV kinks in the energy space, plus the fact that the Gardner transformation also maps $L^2$ solutions
in KdV into $H^1_{loc}$ solutions of modified KdV to conclude that KdV solitons are actually stable in $L^2$.
This approach has since been extended to conclude that multi-soliton solutions of KdV are also stable in
$L^2$, \cite{munoz:2011},  and also that the soliton solution of the nonlinear Schr\"odinger equation is stable in $L^2$, \cite{pelinovsky:2010}.
\item A second advantage of B\"acklund transformation methods is that they can sometimes be used as
 the starting point for a perturbative
argument which yields insight into the behavior of other non-integrable systems.  
Thus, the B\"acklund transformation-based study of the stability
of soliton solutions of the (integrable) Toda-lattice in \cite{mizumachi:2008} served as the basis for a simple proof of stability
of solitary waves in a general class of non-integrable Fermi-Pasta-Ulam models \cite{hoffman:2008}.
\end{itemize}

While this paper will focus on rigorous applications of the B\"acklund transformation method it is 
worth noting that similar ideas have been used
in non-rigorous settings (sometimes in advance of the rigorous applications) to compute explicit approximate solutions
with initial conditions close to solitary waves.  Thus, in \cite{mann:1997}, Mann used a linearized 
B\"acklund transformation to compute the Green's
function for the KdV equation linearized about the soliton solution and then in turn used this to study the evolution
of initial conditions close to the soliton.
Likewise, Tsigaridas, et al \cite{tsigaridas:2005} make a more general study of this same question and apply these ideas
to compute approximate solutions of both the nonlinear Schr\"odinger equation
and KdV equations with the aid of linearized B\"acklund transformations.

The classical view of the B\"acklund transform for the sine-Gordon equation is geometric in nature.  It relates angles between curves of zero curvature on patches of pseudo-spherical surfaces.  As is common in differential geometry, partial differential equations arise.  Here the partial differential equation relates the aforementioned angle as a state variable to the coordinates on the manifold as independent variables.  Thus the geometric relationship between these angles on a pair of psuedo-spherical surfaces manifests in the PDE world as a relationship between a pair of solutions.  

The zero solution is related in this way to a family of monotone front solutions $u(t,x) =u_c(x-ct-\delta)$ which connect $0$ and $2\pi$.  In one physical model, the state variable $u$ in the sine-Gordon equation corresponds to the angle by which an elastic ribbon is twisted from vertical at position $x$ and time $t$.  The front solution obtained from the zero solution via Backlund transform (a maneuver which naively appears to have everything to do with the geometry of pseudo-spherical surfaces and nothing to do with the twisting of elastic ribbons) thus corresponds to an elastic ribbon that has a full twist, or kink, and is commonly called a kink solution.  Applying the Backlund transformation to the kink solution can now produce a solution which is asymptotic to $0$ and $4\pi$ at spatial $\pm \infty$, i.e. has two kinks.  This so-called two-kink solution resolves as $t \to \pm \infty$ to the linear combination of two well-separated kink solutions, each traveling with its own characteristic speed.  Moreover, the characteristic speeds of the kinks are identical at temporal $\pm \infty$ but the phases are allowed to vary.  Thus we can regard the two-kink solution as capturing an interaction in which a fast steep kink overtakes a slow shallow kink with only a phase shift (as opposed to excitation of dispersive modes) to show for the nonlinearity.  Repeated application of the Backlund transform can produce multi-kink solutions which resolve into linear combinations of multiple kink solutions much as multi-soliton solutions resolve into linear combinations of solitons.  

It is well-known to experts that the perspective of the Backlund transformation is a very useful for constructing multi-soliton solutions.  In studying the stability of these kink and multi-kink solutions, however, the theory of dynamical systems is necessarily brought in and from this perspective the classical view of the B\"acklund transform is not entirely natural.  The reason for this is as follows.  The typical strategy of proof in the nascent literature of stability via B\"acklund transform is effectively to conjugate the flow about a soliton or multi-soliton (or the linearization thereabout) with the flow about the zero solution, leveraging the stability of the zero solution to obtain the stability of the soliton or multi-soliton.  The problem from the perspective of dynamical systems is that when conjugating a flow one makes use of a map that acts on the phase space and not on the much larger space of trajectories in the phase space.  One of the key ideas in this paper is to redefine the B\"acklund transform as a map that acts on the phase space.  An orbital stability result for solitons and multi-solitons then follows very quickly from this definition with the aid of well-developed and classical ideas in dynamical systems.

\section{Abstract orbital stability}
\begin{definition}
Let $X$ and $Y$ be open subsets of affine subspaces of Banach spaces and let $\Phi(t) : X \to X$ and $\Psi(t) : Y \to Y$ be semiflows.  
Let $\Lambda$ be a finite dimensional manifold and let $Z$ be a Banach space.
Let $F : X \times Y \times \Lambda \to Z$ be a $C^2$ function such that for each $\lambda \in \Lambda$, $\mathcal{M}_\lambda := F(\cdot,\cdot,\lambda)^{-1}(0)$ is an invariant set for the product flow $\Phi \times \Psi : X \times Y \to X \times Y$.
Assume further 
\begin{itemize}
\item[({\bf H0})] there is some $(\bar{x},\bar{y},\bar{\lambda}) \in X \times Y \times \Lambda$ such that $F(\bar{x},\bar{y},\bar{\lambda}) = 0$.
\item[({\bf H1})] $D_xF(x,y,\lambda) : T_x X \to T_{F(x,y,\lambda)}Z$ is Fredholm and injective whenever $F(x,y,\lambda) = 0$.  
\item[({\bf H2})] $D_{y,\lambda}F(x,y,\lambda) : T_yY \to T_{F(x,y,\lambda)} Z$ is an isomorphism whenever $F(x,y,\lambda) = 0$.
\end{itemize}
Then we say that $F$ {\bf B\"{a}cklund-conjugates} the flows $\Phi$ and $\Psi$.  In the case that $\Phi = \Psi$ we say that $F$ {\bf auto-B\"{a}cklund-conjugates} $\Phi$ with itself.
\end{definition}

\begin{theorem} \label{th:main}
Assume (H0)-(H2).  Let $(\bar{x},\bar{y},\bar{\lambda})$ be given as in (H0) and let $H \subset Y$ be an invariant manifold for $\Psi$ that contains $\bar{y}$ and is stable in sense of Lyapunov:  There is an $\eps_0 > 0$ such that for each $\eps \in (0,\eps_0]$ there is a $\delta = \delta(\eps) > 0$ such that for any $t >0$, we have $d_Y(\Psi(t) y,H) \le \eps$ whenever $d_Y(y,H) \le \delta$.  Assume further that 
\begin{enumerate} 
\item[C1] Given any compact subset $\Lambda_0 \subset \Lambda$ there is a constant $C$ such that for each $(x,y,\lambda) \in F^{-1}(0)$, there is a subspace $X_1$ complementary to the kernel of $D_x F(x,y,\lambda)$ such that 
\[ \left\|\left(\left.D_xF(x,y,\lambda)\right|_{X_1}\right)^{-1}\right\| \le C \]
with the estimate uniform among $y \in Y$ with $d_Y(y,H) < \delta(\eps_0) := \delta_0$, $x \in X$ and $\lambda \in \Lambda_0$ with $F(x,y,\lambda) = 0$.
Furthermore the estimate 
\[ \|D_x F(x_0,y_0,\lambda_0) - D_x F(x_1,y_1,\lambda_1)\| \le C(d_X(x_0,x_1) + d_Y(y_0,y_1) + d_\Lambda(\lambda_1,\lambda_0)) \]
holds uniformly among $x_0, x_1, y_0, y_1, \lambda_0, \lambda_1$ with $d_Y(y_j,H) < \delta_0$ and $(x_j,y_j,\lambda_j) \in F^{-1}(0)$.

\item[C2] the norm of $(D_{y,\lambda}F(x,y,\lambda))^{-1}$ is bounded above uniformly among $(x,y,\lambda)$ such that $d_Y(y,H) < \delta(\eps_0)$ and $\lambda \subset \Lambda_0$ compact and $F(x,y,\lambda) = 0$.  Furthermore, $D_{y,\lambda}F(x,y,\lambda)$ is uniformly Lipschitz among $(x,y,\lambda) \in F^{-1}(0)$ such that $d_Y(y,H) < \delta_0$

\item[C3] the norm of $D_x F(x,y,\lambda)$ is bounded above uniformly among $(x,y,\lambda)$ with $d_Y(y,H) < \delta(\eps_0)$, with $\lambda \subset \Lambda_0$ compact and $F(x,y,\lambda) = 0$.
\end{enumerate}
 
Then there is an invariant manifold $M$ for $\Phi$ containing $\bar{x}$, a function $\lambda^* : M \to \Lambda$, a decomposition of $M$ into invariant manifolds $M^\lambda = (\lambda^*)^{-1}(\lambda)$, as well as a constant $C$ such that 

\[  d_X(\Phi(t)x,M^{\lambda^*(x)}) \le C\eps \]
whenever $d_X(x,M) \le \frac{1}{C}\delta$.  Moreover, $M^\lambda$ is precisely the set of $x \in X$ such that $F(x,y,\lambda) = 0$ for some $y \in H$.
\end{theorem}
\begin{proof}
Let $(\bar{x},\bar{y},\bar{\lambda})$ be given as in (H0).  It follows from 
(H2) 
and the implicit function theorem that there are smooth functions $y^*$ and $\lambda^*$ mapping a neighborhood of $\bar{x}$ to neighborhoods of $\bar{y}$ and $\bar{\lambda}$ respectively such that $F(x,y^*(x),\lambda^*(x)) = 0$.  Furthermore, these functions are unique in that if $(x,y,\lambda)$ is close to $(\bar{x},\bar{y},\bar{\lambda})$ and $F(x,y,\lambda) = 0$, then $y= y^*(x)$ and $\lambda = \lambda^*(x)$.  

We claim that in addition, the functions $y^*$ and $\lambda^*$ can be extended to some maximal domain such that on this domain the range of $y^*$ contains a $\delta$-neighborhood of $H$.  To establish this claim, first note that (H2) allows us to enlarge the domain on which $y^*$ and $\lambda^*$ is defined by applying the implicit function theorem with base point $(x,y^*(x),\lambda^*(x))$ for any $x$ in the domain of $y^*$ and $\lambda^*$.  Furthermore, one can observe from the proof of the implicit function theorem that diameter of the neighborhoods on which the implicit functions $y^*$ and $\lambda^*$ are defined can be taken to be $2 \|D_x F(x,y,\lambda)^{-1}\| \mathrm{Lip} D_x F$.  In light of condition (C1), this diameter is uniform among $(x,y,\lambda) \in F^{-1}(0)$ for which $d_Y(y,H) < \delta_0$.  This establishes the claim: by repeatedly applying the implicit function theorem we can enlarge the domain of $y^*$ and $\lambda^*$ sufficiently so that the range of $y$ covers a neighborhood of $H$.

Note that \[ \left( \ba{c} D_x y^*(x) \\ D_x \lambda^*(x) \ea \right) = -D_{y,\lambda} F(x,y^*(x),\lambda^*(x))^{-1} F_x(x,y^*(x),\lambda^*(x)), \]
and hence the implicitly defined functions $y^*$ and $\lambda^*$ are Lipschitz on any set for which $x \mapsto \|D_{y,\lambda}F(x,y^*(x),\lambda^*(x))^{-1}\|$ and $x \mapsto \|F_x(x,y^*(x),\lambda^*(x))\|$ enjoy a uniform bound.
Define $M = (y^*)^{-1}(H)$ and with the notation $H_\delta = \{y \; | \; d_Y(y,H) < \delta\}$ define for $\delta$ sufficiently small $M_\delta = (y^*)^{-1}(H_\delta)$.  It follows from conditions (C2) and (C3) that $y^*$ and $\lambda^*$ are Lipschitz on $M_\delta \subset X$ for $\delta \le \delta(\eps_0)$.

Let $T_{\bar{x}}X = \bar{K} \times \bar{X_1}$ be a Lyapunov-Schmidt decomposition of $T_{\bar{x}}X$ subordinate to $D_x F(\bar{x},\bar{y},\bar{\lambda})$.  Here $\bar{K}$ denotes the kernel and $\bar{X}_1$ is chosen as in (C1).  It follows from (H1) that there is a smooth implicitly defined function $\overline{x_1^*}$ taking a neighborhood of $(\bar{y},\bar{\lambda},0) \in Y\times \Lambda \times K$ to $\bar{X_1}$ such that $F(\overline{x_1^*}(y,\lambda,k)+k,y,\lambda) = 0$ for any $k$ in the given neighborhood of $\{0\} \subset K$.  Since the range of $y^*$ contains $H_\delta$, it follows that $y = y^*(\overline{x_1^*}(y,\lambda,k)+k)$ and more specifically that for $y \in H$ we have $\overline{x_1^*}(y,\lambda,k)+k \in M$ whenever this quantity is defined.

Because of (H1) the point $(\bar{x},\bar{y},\bar{\lambda})$ is not distinguished among points in $F^{-1}(0)$.  Thus given any $(x,y,\lambda) \in F^{-1}(0)$ there is a similar Lyapunov-Schmidt decomposition $T_{x}X = K \times X_1$ and a similar implicitly defined function $x_1^*$.  It follows from (C1) that $x_1^*$ has a Lipschitz constant which is uniform in the choice of base point $(x,y,\lambda)$ for the implicit function theorem.

Let $C$ denote the Lipschitz constant of $y^*$.  Since we have assumed in the statement of the theorem that $x$ is $\delta / C$-close to $M$ it follows that $y^*(x)$ is $\delta$-close to $H$.  Recall that this is the neighborhood of Lyapunov stability for $H$ corresponding to the given small number $\eps$: $d_Y(\Psi(t)y^*(x),H) < \eps$.  Since the pair $(x,y^*(x))$ lies on the invariant manifold $\mathcal{M}_{\lambda^*(x)} = F(\cdot,\cdot,\lambda^*(x))^{-1}(0)$ it follows that 
$$F(\Phi(t)x,\Psi(t)y^*(x),\lambda^*(x)) = 0\ ,$$
 hence that $\Phi(t)x = x_1^*(\Psi(t)y^*(x),\lambda^*(x),k)+k$ for one of the local functions $x_1^*$.  
 
We now establish the Lyapunov stability of $M$:
\[ \ba{lcl} d_X(\Phi(t)x,M) & \le & d_X(\Phi(t)x,x_1^*(y^*(x),\lambda^*(x),k)+k) \\ \\
& = & d_X(x_1^*(\Psi(t)y^*(x),\lambda^*(x),k)+k,x_1^*(y^*(x),\lambda^*(x),k)+k) \\ \\
& \le & \mathrm{Lip}x_1^* d_Y(\Psi(t)y^*(x),y^*(x)) \\ \\
& \le & \eps \mathrm{Lip}x_1^*
\ea
\]
In the first line we have used our characterization of $M$.  In the second line we have used that $F^{-1}(0)$ is invariant for the product semiflow.  In the third line we have used that $x_1^*$ is Lipschitz and in the fourth line that $H$ is Lyapunov-stable for the semiflow $\Psi$.
\end{proof}

\section{Examples of orbital stability}
\subsection{Sine-Gordon equation}

As a first application of Theorem \ref{th:main}, we consider the orbital stability of the kink solutions of the Sine-Gordon
equation.  While the stability of these solutions is not surprising and could probably be proved using the energy
methods discussed above, it gives a simple illustration of our approach.  Furthermore, as we indicate at the end of this
section,  we suspect that with some additional work this method will also yield the stability of multi-kink
solutions for this equation.

The classical B\"acklund transform for the Sine-Gordon equation relates two solutions $\bar{u}$ and $\bar{u}'$ by the pair
of equations
\begin{eqnarray} \label{eq:SG-BT}
\bar{u}_x-\bar{u}_t &=& \bar{u}_x'-\bar{u}_t' + 2a \sin(\frac{\bar{u}+\bar{u}'}{2}) \\ \nonumber
\bar{u}_x + \bar{u}_t &=& - \bar{u}_x' - \bar{u}_t' + \frac{2}{a} \sin(\frac{\bar{u}-\bar{u}'}{2})
\end{eqnarray}

If we introduce phase space variables $u=\bar{u}$, $v= \bar{u}_t$ and $u'=\bar{u}'$, $v'= \bar{u}_t'$, we see that
the B\"acklund transform can be written as
\begin{eqnarray}\label{eq:Fdef}
F(u,v,u',v',a) = \left( \ba{c} u_x +v' - a\sin(\frac{u+u'}{2}) - \frac{1}{a}\sin(\frac{u-u'}{2}) \\  v + u_x' - \frac{1}{a}\sin(\frac{u-u'}{2}) + a\sin(\frac{u+u'}{2}) \ea \right) = 0 
\end{eqnarray}
and it is to this function that we apply Theorem \ref{th:main}

Recall that the B\"acklund transform for the Sine-Gordon equation maps the zero solution to the $1$-kink solution
and then successively maps the $k$-kink to the $k+1$-kink, for any positive integer $k$ \cite{lamb:1971}.  With this in mind, 
let $\bar{X} = \{(u,v) \; | \; \sin(\frac{u}{2}) \in H^1 \mbox{ and } v \in L^2 \mbox{ and } u(-\infty) = 0\}$.  Consider the decomposition $\bar{X} = \cup_{k = -\infty}^\infty X_k$ where $X_k = \{(u,v) \in \bar{X} \; | \; u(\infty) = 2\pi k\}$.  \footnote{Note that the fact that $\sin(u/2) \in H^1$ implies that the jump in $u$ from $-\infty$ to $\infty$ is an integer multiple of $2\pi$.}  Let $\Phi(t)=\Psi(t)$ denote the time $t$ map for the sine-Gordon equation $u_{tt} = u_{xx} - \sin u$ and let $Z = L^2 \times L^2$.  Given the properties of the $k$-kink
solution, it is natural to study its evolution in the space $X_k$.  For the time being, since we want to concentrate on
the $1$-kink solution we will focus on the spaces $X_0$ and $X_1$, and consider the function:
\begin{equation}
F: X_{1} \times X_0 \times (0,1) \to Z
\end{equation}
We now have:
\begin{theorem}\label{thm:one-kink} 
Let $\bar{u}$ be a $1$-kink solution for the sine-Gordon equation, let $\eps > 0$ be given and let $u^0$ be 
 $\eps$-close to $\bar{u}(t)$ in $H^1 \times L^2$  for time $t=0$.  (i.e. $(u_0 - \bar{u}(0))$ is small
 in $H^1 \times L^2$).  Let $u^t$ denote the time-evolution of $u^0$.
 Then for all time,  $u^t$ remains $\sqrt{\eps}$-close to some $1$-kink solution with to the speed of $\bar{u}$ 
 \end{theorem}

\begin{proof}  We first  check that hypotheses $(H0)$, $(H1)$, and $(H2)$ are satisfied.
We take $\bar{y}=(u',v')=(0,0)$, and then we can solve explicitly for those points $\bar{x}=(u,v)$ for which $F$ is zero and we
find $u= 4 \arctan(\exp( a x + \delta))$, $v= 4  a v \exp( a x + \delta)/(1+\exp(2(a x + \delta))$, where
$v$ and $a$ are related by $a \sqrt{1-v^2} = 1$.  (Note that this calculation insures that $(H0)$ is satisfied.)
However, we will use this explicit form of the kink solution very little, in order
 to set the stage for our discussion of the stability of the $k$-kink solution later in this section.

Now differentiate $F$ with respect to $(u,v)$ to obtain:
\begin{equation}
D_{(u,v)}F(u,v,u',v',,a) = \left( \ba{cc} \partial_x - \frac{a}{2}\cos(\frac{u+u'}{2}) - \frac{1}{2a}\cos(\frac{u-u'}{2}) & 0 \\ \\ 
\partial_x -\frac{1}{2a}\cos(\frac{u+u'}{2}) + \frac{a}{2}\cos(\frac{u-u'}{2}) & 1 \ea \right) 
\end{equation}
To invert this operator we need to solve the system of ODE's:
\begin{equation}
\left( \ba{cc} \partial_x - \frac{a}{2}\cos(\frac{u+u'}{2}) - \frac{1}{2a}\cos(\frac{u-u'}{2}) & 0 \\ \\ 
\partial_x -\frac{1}{2a}\cos(\frac{u+u'}{2}) + \frac{a}{2}\cos(\frac{u-u'}{2})   & 1 \ea \right) \left( \ba{c} \phi \\ \psi \ea \right) = 
\left( \ba{c} f \\ g\ea \right)
\end{equation}
Note that the operator is lower triangular, and  the second row gives $\psi $ in terms $\phi_x$,   a bounded
 (invertible) multiplication operator acting on $\phi $, and  $g$. 

Thus, we focus on the first row 
which gives $\phi $ as the solution to a first order non-autonomous ODE (in $x$) with inhomogeneous term given by 
$f$.    More precisely, we must solve
\begin{equation}\label{eq:dxF}
\partial_x \phi - \left( \frac{a}{2}\cos(\frac{u+u'}{2}) + \frac{1}{2a}\cos(\frac{u-u'}{2}) \right)  \phi = f\ .
\end{equation}

We analyze this equation, and similar equations below with the aid of the following lemma.
\begin{lemma} \label{lem1}
Consider the ODE 
\begin{equation}\label{eq:inhom}
u_x - \alpha(x)u  =  f(x)\ .
\end{equation}
\begin{enumerate}
\item Assume that $\alpha_{\pm} = \lim_{x\to \pm \infty} \alpha(x) $ 
are defined with $\alpha_{-}> 0 > \alpha_{+} $.  Assume further that
$ \int_0^{\infty} |\alpha(t) - \alpha_+ |dt  < \infty$, and $  \int_{-\infty}^{0} |\alpha(t) - \alpha_- | dt  < \infty$
 Then there exists a constant $C_{\alpha}$, such that the unique solution of  \eqref{eq:inhom}
with  $u(x_0) = 0$ satisfies $\|u\|_{H^1} \le C_{\alpha}  \|f\|_{L^2}$.
\item Assume that  $\alpha_{-}  < 0 < \alpha_{+} $, 
with $|\alpha(x)| \le \kappa < \infty$ and that $\int_\R f(t) \phi(t)dt = 0$ for the unique (up to contant multiple),
non-zero,  bounded  
$\phi$ solving  the adjoint ODE $\dot{\phi} = -\alpha(t)\phi$.  Then there is a unique choice of $u_0$ for which 
$u \in L^2$ and for this choice of $u_0$ we have $\|u\|_{H^1} \le C_{\alpha}  \|f\|_{L^2}$.  
\end{enumerate}
\end{lemma}

\begin{remark}  This lemma is a simple and explicit example of the relationship between
Fredholm properties of operators and exponential dichotomies, which has been very useful
in the theory of dynamical systems \cite{palmer:1984}.  In particular, the fact that we require $\alpha(x)$ to
converge to its limiting values in $L^1$ is a very natural assumption in this context.
\end{remark}

\begin{proof} (of Lemma) Define $\mu(x) = \exp(- \int_{x_0}^x \alpha(t) dt )$.
Then the unique solution of \eqref{eq:inhom} with $u(x_0) = u_0$ is
\begin{equation}\label{eq:inhom_soln}
u(x) = u_0/\mu(x) + \int_{x_0}^x \frac{ \mu(y) }{\mu(x) } f(y) dy\ .
\end{equation}
Write $u = u^> + u^<$, where $u^>(x) = u(x)$ for $x \ge x_0$, and zero otherwise, and 
$u^<(x) =  u(x)$ for $x < x_0$, and zero otherwise.  Then $\| u \|_{L^2}^2 = \| u_> \|_{L^2}^2 + \| u^< \|_{L^2}^2$.
We will bound $\| u_> \|_{L^2}^2$ and leave the estimate on $\| u^< \|_{L^2}^2$ as an exercise.

We now consider specifically the situation in Case 1, where  $\alpha_{-}> 0 > \alpha_{+} $.  Define
$h^{>}(x) = |h(x)|$ if $x> x_0$ and $h^{>}(x) = 0$ for $x \le x_0$.  Likewise define
${\mathcal E}^{>}(x) = \exp(\alpha_{+} x)$ if $x>0$ and zero otherwise.  Then we can estimate
\begin{eqnarray}\nonumber
|u^{>}(x)| &=&  |\int_{x_0}^x e^{\int_y^x \alpha(t)dt} h(y) dy | \\ \nonumber
&=&   |\int_{x_0}^x e^{\alpha_{+}(x-y) + \int_y^x (\alpha(t)-\alpha_{+}) dt} h(y) dy | \\ \nonumber
& \le & C^+_{\alpha} \int_{x_0}^x e^{\alpha_{+} (x-y)} |h(y)| dy = C^+_{\alpha} \int_{-\infty}^{\infty} {\mathcal E}_+(x-y) h^{>}(y) dy
\end{eqnarray}
From this last expression we immediately obtain $\| u^{>} \|_{L^2} \le \tilde{C}^+_{\alpha} \| {\mathcal E}_+ \|_{L^1} \| h^{>} \|_{L^2}$,
from Young's inequality.  The $L^2$ norm of the derivative of $u^{>}$ can be estimated in a similar fashion which
completes the estimate of the $H^1$ norm, and the proof of Case 1.

Now turn to Case 2.  Rewrite \eqref{eq:inhom_soln} as 
\begin{equation}
\mu(x) u(x) = u_0 + \int_{x_0}^x \mu(y) f(y) dy
\end{equation}
The assumptions on $\alpha$ imply that $\mu(x) \to 0$ as $|x| \to \infty$ and thus, in order for the solution $u(x)$
to be bounded we must have
\begin{equation}
u_0 + \int_{x_0}^{\infty} \mu(y) f(y) dy = u_0 + \int_{x_0}^{-\infty} \mu(y) f(y) dy =0\ ,
\end{equation}
which uniquely defines $u_0$ provided
\begin{equation}
0 = \int_{x_0}^{\infty} \mu(y) f(y) dy - \int_{x_0}^{-\infty} \mu(y) f(y) dy = \int_{-\infty}^{\infty} \mu(y) f(y) dy
\end{equation}
Note that $\mu(x)$ is the solution of the adjoint ODE, and hence the hypothesis of Case 2 is satisfied.
The resulting bound on the norm of $u$ then follows as in Case 1.
\end{proof}

We now apply the Lemma to \eqref{eq:dxF} where we see that
\begin{equation}
\alpha(x) = \frac{a}{2} \cos(\frac{u+u'}{2}) + \frac{1}{2a} \cos(\frac{u-u'}{2}) 
\end{equation}
Recalling that $u'\in H^1$ and that $u$ is an $H^1$ perturbation of the kink, we see that $u\pm u'$ will approach
their limits as $x\to \pm \infty$ in $L^2$ and hence that $ \cos(\frac{u\pm u'}{2})$ will approach their limits in $L^1$ as
required.  Furthermore, $\alpha_{\pm} = \mp (\frac{a}{2} + \frac{1}{2a})$, so we are in Case 1 of the Lemma.
Thus, \eqref{eq:dxF} can be solved for any $f \in L^2$, which gives the invertibility of $D_{u,v} F$ and verifies
$(H1)$.

Next consider hypothesis $(C1)$ of Theorem \ref{th:main}.  Note that from the calculation above, we see that 
$D_{(u,v)}F$ does have a one-dimensional kernel spanned by 
\[ \phi(x) = \mu(x) \mbox{ and } \psi(x) =
\left( (\frac{1}{2a}) \cos(\frac{u+u'}{2}) - \frac{a}{2} \cos(\frac{u-u'}{2}) \right) \mu(x) - \mu'(x).\]  Consider the
subspace orthogonal to this kernel.  Note that we need a uniform bound on the inverse of $D_{(u,v)}F$ only
for $(u',v')\in H = \{ (0,0) \}$ and for $a$ in some compact subset of $(0,1)$.  A bound on the inverse is
easily derived from the proof of the lemma and is proportional to the constant $C_{\alpha}$ which we again
see from the proof of the lemma is determined by $ \int_{x_0}^{\pm \infty} | \alpha(t) - \alpha_{\pm}| dt$.
Since $u'=0$, we know that $u(x)$ is given by a $1$-kink solution of the Sine-Gordon equation and 
hence $\alpha(x) =  \frac{a}{2} \cos(\frac{u}{2}) + \frac{1}{2a} \cos(\frac{u}{2}) $.  If we choose the point $x_0$ to be
the midpoint of the kink, then it is easy to show that the quantities $ \int_{x_0}^{\pm \infty} | \alpha(t) - \alpha_{\pm}| dt$,
and hence the constants $C_{\alpha}$, can be bounded uniformly for all kinks with parameter $a$ in some
compact subinterval of $(0,1)$.  To verify hypothesis $(C1)$ it remains only to check that the derivative $D_{(u,v)}F$ is Lipschitz, but this follows from the fact that $\cos$ regarded as a function $\R \to \R$ is Lipschitz together with the fact that the operator norm of a multiplication operator is bounded by the $L^\infty$ norm of the the function by which it mutiplies.

We now turn to $(H2)$ and $(C2)$.  In this case, we must solve the equations
\begin{equation}
D_{(u',v',a)} F(u,v,u',v',a) \left( \ba{c} \phi \\ \psi \\ \delta a \ea \right) = 
\left( \ba{c} f \\ g\ea \right)\ ,
\end{equation}
where
\begin{equation}
D_{(u',v',a)} F(u,v,u',v',a) = \left( \ba{ccc} -\frac{a}{2}\cos(\frac{u+u'}{2}) + \frac{1}{2a} \cos(\frac{u-u'}{2}) & 1 & -\sin(\frac{u+u'}{2}) + 
a^{-2}\sin(\frac{u-u'}{2}) \\ \\ \partial_x + \frac{1}{2a}\cos(\frac{u-u'}{2}) + \frac{a}{2}\cos(\frac{u+u'}{2}) & 0 & a^{-2} \sin(\frac{u-u'}{2}) + \sin(\frac{u+u'}{2}) \ea \right)
\end{equation}

We can solve the first equation to find $\psi$ in terms of $\phi$, $f$, $g$ and $a$.  Thus, we focus on the second
equation
\begin{equation}\label{eq:dyF}
\partial_x \phi +  \left( \frac{a}{2}\cos(\frac{u+u'}{2}) + \frac{1}{2a}\cos(\frac{u-u'}{2}) \right)  \phi = g - b(x) \delta a
\end{equation}
where $b(x) = a^{-2} \sin(\frac{u-u'}{2}) +\sin(\frac{u+u'}{2})$  This is remarkably similar to equation \eqref{eq:dxF}
except that the sign in front of the non-autonomous term $ \left( \frac{a}{2}\cos(\frac{u+u'}{2}) + \frac{1}{2a}\cos(\frac{u-u'}{2}) \right) $
has changed.  This means that we are in Case 2 of the Lemma, rather than Case 1, and in order to solve the equation we
must check that the right hand side of \eqref{eq:dyF} is orthogonal to the solution of the adjoint ODE.
In this case, one can check that the solution of the adjoint ODE is $\mu(x) = \exp( \int_{x_0}^x  \alpha(t) dt)$, where
in this case $\alpha(x) = \left( \frac{a}{2}\cos(\frac{u+u'}{2}) + \frac{1}{2a}\cos(\frac{u-u'}{2}) \right) $.  We can 
insure that the RHS of \eqref{eq:dyF} is orthogonal to $\mu$ by picking $\delta a$ appropriately,
provided $\int_{-\infty}^{\infty} b(x) \mu(x) dx \ne 0$.

So far, we have not found any way of demonstrating that this integral is non-zero for an arbitrary choice of $u \in X_1$
and $u' \in X_0$.  However, if we take $u'=0$ and $u$ equal to a $1$-kink, we have
$b(x) = 2(1+a^{-2}) \exp(ax)/(1+\exp(2ax)) > 0$.  Likewise, $\mu(x) > 0$ for all $x$, so $\int_{-\infty}^{\infty} b(x) \mu(x) dx \ne 0$.
Since both $b$ and $\mu$ depend smoothly on $u$ and $u'$, this condition will also hold for
all $u$ near the $1$-kink and all $u'$ near the zero-solution.  Thus, Theorem \ref{thm:one-kink}, holds on such
a neighborhood.  This verifies hypothesis $(H2)$.  To check $(C1)$ we need only derive a uniform estimate the solution
$\phi$ of \eqref{eq:dyF} for $u'=0$ and $a$ in some compact subinterval of $(0,1)$.  This follows in a very
similar fashion to estimates on solutions of \eqref{eq:inhom_soln}, and we leave the details as an easy exercise.

\end{proof}

\begin{remark}
We note that there is a natural path to attempt to build on the preceding result to establish the stability of an arbitrary
$k$-kink solution.  It is known that the B\"acklund transformation \eqref{eq:SG-BT} links the $k$-kink to the $k+1$-kink.
Thus, we can repeat the above proof, this time considering
\begin{equation}
F:X_2\times X_1 \times (0,1) \to Z 
\end{equation}
and considering the the base point of our theorem $(u',v')$ to be a $1$-kink.  Then the manifold $H$ is the family
of $1$-kinks.  If we then consider the linearizations $D_{(u,v)} F$ and $D_{(u',v',a)} F$, the verification of
$(H1)$, $(H2)$, $(C1)$ and $(C2)$ proceed much as above.  The only points that need to be checked are the
uniform estimates on the inverses.  These require uniform estimates on the analogues of \eqref{eq:dxF} and
\eqref{eq:dyF}. 
In our estimates of the stability of the $1$-kink, we had the freedom to choose the point $x_0$ to be the center of the kink.
This made it simple to establish uniform estimates.  For multi-kink solutions, there is no such distinguished point  and we
need an analysis of the form of the $2$-kink solution to show that for
large time, the solution of these equations can be treated essentially  by regarding the
$2$-kink as a sum of two $1$-kink solutions which were estimated above.  
This program is carried out in detail to establish the stability of the
multi-soliton solutions of the Toda lattice in \cite{benes:2011}.
Once the stability of the $2$-kink
solution is established, it can be used in conjunction with the B\"acklund transformation to establish the
stability of the $3$-kink, and so-on and so-forth. \end{remark}

\subsection{The Toda Lattice}
As a second application of Theorem $\ref{th:main}$ we study the orbital stability of the multi-soliton solutions in the Toda Lattice 
\be \dot{q}_j = p_j; \quad \dot{p}_j = e^{q_{j-1}-q_j} - e^{q_j-q_{j+1}} \label{eq:todaev} \ee
posed in the energy space $\ell^2 \times \ell^2$.  Here $q$ and $p$ can be regarded as the position and momentum, respectively, of the jth particle in an infinite chain where neighboring particles resist compression quite strongly but resist extension only weakly.

Because of the lattice discreteness, some of the quantities that are most easily used to obtain multi-soliton solutions as constrained minimizers of a Lyapunov function in the PDE case are no longer conserved and so the Lyapunov function approach is not easily extended.  For single solitons more detailed stability results have been established, specifically orbital stability with asymptotic phase, and moreover asymptotic stability in a weighted space \cite{mizumachi:2009}.  
The techniques used in \cite{mizumachi:2009} were a combination of the B\"acklund approach we take here  with the dispersive approach mentioned in the introduction .  Using a linearized version of the B\"acklund transform, asymptotic stability of multi-solitons has also been obtained, albeit in an exponentially weighted space \cite{benes:2011}.  

The B\"{a}cklund transform for the Toda lattice, written in our framework, is 
\be
F_j(q,p,q',p',\kappa) = \left( \ba{l} p_j + e^{-(q_j'-q_j-\kappa)} + e^{-(q_j-q'_{j-1}+\kappa)} - 2\cosh \kappa \\ p_j' + e^{-(q_j'-q_j-\kappa)} + e^{-(q_{j+1}-q_j' + \kappa)} - 2\cosh \kappa \ea \right) 
\label{eq:todaBT}
\ee
where $\kappa$ is a real parameter.  As with the sine-Gordon equation considered above, the zero and one-kink solutions (and more generally the $m$- and $m+1$-kink solutions) are related via $\eqref{eq:todaBT}$ with the parameter $\kappa$ controlling the amplitude $(2\kappa)$ and speed $\frac{\sinh \kappa}{\kappa}$ of the additional kink.   

Let $\Phi = \Psi$ be the propagator for the Toda lattice.  Denote the $m$-kink solution with amplitude parameter $\kappa_1,\cdots \kappa_m$ and phase parameters $\delta_1, \cdots \delta_m$ by $(q^{\kappa_1,\cdots \kappa_m,\delta_1,\cdots \delta_m},p^{\kappa_1,\cdots \kappa_m,\delta_1,\cdots \delta_m})$ or more concisely $(q^m,p^m)$ when the parameters are understood.  Given an $m$-kink solution $(q^m,p^m)$ and an $m+1$-kink solution $(q^{m+1},p^{m+1})$ related via B\"acklund transform with parameter $\kappa$: $F(q^{m+1},p^{m+1},q^m,p^m,\kappa) = 0$, define the affine spaces $X_m$ and $X_{m+1}$ to be the space of $\ell^2$ perturbations of $(q^m,p^m)$ and $(q^{m+1},p^{m+1})$ respectively.

\begin{theorem} \label{th:toda}
The Toda-m-soliton is orbitally stable in $\ell^2$ in the sense of Lyapunov: Let $M$ denote the $2m$-dimensional manifold of $m$-soliton states with phases free to vary in $\R$ and with amplitudes constrained to any compact set.  Let $(Q^m,P^m)$ denote a point on $M$.  For each $\eps > 0$ there is a $\delta > 0$ such that whenever $(q,p) \in \ell^2 \times \ell^2$ satisfies $\|(Q^m,P^m)-(q,p)\|_{\ell^2 \times \ell^2} < \delta$, then its forward evolution $\Phi(t)(q,p)$ under $\eqref{eq:todaev}$ satisfies $d(\Phi(t)(q,p),M) < \eps$.
\end{theorem}

\begin{proof}
The proof proceeds by induction.  At each stage of the induction we apply Theorem \ref{th:main}.  At the $k^{th}$ stage of the induction, the invariant manifold $H$ is a particular $1$-manifold corresponding to the temporal evolution of a particular $k-1$-soliton while $M$ is a particular $2$-manifold corresponding to the temporal evolution of a one-parameter family of $k$-solitons with the parameter corresponding to the initial phase of the additional soliton.  The inductive hypothesis is used not to verify the hypotheses (H0)-(H2) and (C0)-(C2) which are established by hand at each stage of the induction, but rather to verify that $H$ is Lyapunov-stable.  This is a natural use of the inductive hypothesis because the $H$ used for the $k^{th}$ stage of induction is a submanifold of the $M$ used for the $(k-1)^{st}$ stage of the induction and hence Lyapunov-stability is guaranteed by Theorem \ref{th:main}.

We first verify (H0).  In the base case we set $(q',p') = (0,0)$ and solve the equation $F(q,p,0,0,\kappa) = 0$ for $q_j = \log\frac{\cosh(\kappa j + \gamma)}{\cosh(\kappa(j+1) + \gamma)}$ and $p_j = e^{\kappa}\left(\frac{\cosh(\kappa j + \gamma)}{\cosh(\kappa j + \gamma + \kappa)} -1 \right) + e^{-\kappa}\left( \frac{\cosh(\kappa j + \gamma + \kappa)}{\cosh(\kappa j + \gamma)}-1\right)$.  More generally, we set $(q',p') = (q^m,p^m)$ and solve $F(q,p,q^m,p^m,\kappa) = 0$ for $(q,p) = (q^{m+1},p^{m+1})$.  To do this computation from scratch is a significant undertaking; we rely on the early literature in the history of the Toda lattice \cite{wadati:1975}.

Note that in this formulation the differences in the asymptotic values satisfy $q'_{-\infty} -q_{-\infty} = 0$ and $q'_\infty - q_\infty = -2\kappa$.  We now check the hypotheses (H1)-(H2).  

To check (H1) we differentiate $F$ and obtain
\[ D_{(p,q)} F(p,q,p',q',\kappa) = \left( \ba{ll} e^{-(q'-q-\kappa)}-e^{-(q-q_-'+\kappa)} & I  \\ \\ e^{-(q'-q-\kappa)}-e^{-(q_+-q'+\kappa)}S & 0 \ea \right). \] Here $S$ is the shift operator $(Sq)_j = q_{j+1}$ and the symbols $q_\pm$ denote the shifted sequences $S^{\pm 1}q$.  We must study 
\[ \left( \ba{ll} e^{-(q'-q-\kappa)}-e^{-(q-q_-'+\kappa)} & I  \\ \\ e^{-(q'-q-\kappa)}-e^{-(q_+-q'+\kappa)}S & 0 \ea \right) \left( \ba{c} \phi \\ \psi \ea \right) = \left( \ba{c} f \\ g \ea \right), \]
obtaining solvability conditions for $\phi$ and $\psi$ in terms of $f$ and $g$.

Note that $\phi$ is given as a linear combination of $f$ and the action of a multiplication operator (bounded from $\ell^2 \to \ell^2$) on $\psi$.  Thus we restrict attention to the second row, which is a first order linear difference equation
\[ \psi_+ = e^{-(2q'-q-q_+-2\kappa)}\psi + e^{(q_+-q'+\kappa)}g \]
This equation is of the form $\psi_{j+1} = \alpha_j\psi_j + f_j$ which can be solved explicitly via a summing factor.  A discrete analog of Lemma $\ref{lem1}$ holds with the relevant numbers now $|\alpha_{\pm \infty}|$ rather than $\mathrm{sgn}(\alpha_{\pm \infty})$:

\begin{lemma}
Consider the first order linear recursion $u_{n+1} -\alpha_n u_n = f_n$ and suppose that the limits $\alpha_\pm = \lim_{n \to \pm \infty} \alpha_n$ exist.
\begin{enumerate}
\item Suppose that $|\alpha_{-\infty}| > 1 > |\alpha_\infty|$ and that $\sum_{n=1}^\infty |\alpha_n - \alpha_+| < \infty$ as well as $\sum_{n = -\infty}^0 |\alpha_n - \alpha_-| < \infty$.  Then there is a constant $C_\alpha$ depending on $\sum_{n=0}^{\pm \infty} |\alpha_n - \alpha_{\pm}|$ such that the unique solution $u$ with $u_0 = 0$ satisfies $\|u\|_{\ell^2} \le C_\alpha \|f\|_{\ell^2}$.
\item Suppose that $|\alpha_{-\infty} < 1 < |\alpha_\infty|$ and that $\sum_{n=1}^\infty |\alpha_n - \alpha_+| < \infty$ as well as $\sum_{n = -\infty}^0 |\alpha_n - \alpha_-| < \infty$.  Then there is a constant $C_\alpha$ depending on $\sum_{n=0}^{\pm \infty} |\alpha_n - \alpha_{\pm}|$ such that if  $f$ satisfies $\sum_{n \in \Z} f_n \phi_n = 0$ for $\phi$ the unique (up to scalar multiple) solution of the adjoint equation $\phi_{n-1} - \alpha_n\phi_n = 0$ then there is a unique choice of $u_0$ such that $u \in \ell^2$ and for this choice $\|u\|_{\ell^2} \le C_\alpha \|f\|_{\ell^2}$
\end{enumerate}
\end{lemma}

Its proof is similar to the proof of Lemma $\ref{lem1}$ and can be regarded as a consequence of the theory of exponential dichotomies or as an exercise for the reader.  

We continue now with the proof of Theorem \ref{th:toda}, computing $\alpha_{\pm} = e^{-2(q'_{\pm \infty} - q_{\pm \infty} - \kappa)} = e^{\mp 2 \kappa}$.  To check the hypotheses of the lemma we must show that $\alpha - \alpha_{\pm}$ is summable over $\pm \mathbb{N}$.  To that end, we compute $a - a_+ = e^{-2\kappa}\left(e^{-(2q'-q'q_+ - 4\kappa)} - 1\right) \approx q' - (q + 2\kappa) + q' - (q_+ + 2\kappa)$.  It follows from Lemma 2.1 in \cite{benes:2011} together with the fact that $q_+-q$ is exponentially localized \cite{PF1}, that this quantity is in $\ell^1$ and moreover approaches $2\kappa$ in $\ell^1$ exponentially fast as $t \to \infty$ when $q$ and $q'$ are $m$ and $m-1$ soliton solutions respectively.  We remark that this is one place where the restriction to a compact set of $\kappa$ is necessary.  These estimates are not uniform in the limit as any of the wave speeds goes to zero. 

The coefficients $\alpha_{\pm}$ satisfy $\alpha_{-} > 1 > \alpha_+$ and thus we are in case (1) of the lemma.  This establishes (H1).  In the base case this establishes the uniform bound in (C1) as well; after all uniform bounded are not hard to obtain on $H \times M \cong \{(0,0)\} \times \R / \Z$.  The constant $C_\alpha$ given in the lemma depends only upon $\sum_{n > 0} \alpha_n - \alpha_+$ which has a limit in $\ell^1$ as $t \to \infty$ and hence is bounded uniformly as $t \to \infty$.  This establishes that the constant $C_\alpha$ is bounded along any trajectory in $H \times M \cap F^{-1}(0)$ that corresponds to the temporal evolution of the $k$-soliton state under $\eqref{eq:todaev}$, i.e. it establishes the uniform bound in (C1) when the manifold $H$ in the theorem is the orbit of a $k$-soliton state.  To verify (C1) it only remains to check that the derivative of $F$ is Lipschitz, but this is immediate.

We now check (H2)-(C2).  We compute
\[ D_{q',p',\kappa}F(q,p,q',p',\kappa) = \left( \ba{ccc} e^{-(q-q'_-+\kappa)}S^{-1} - e^{-(q'-q-\kappa)} & 0 & e^{-(q'-q-\kappa)}-e^{-(q-q_-'+\kappa)} -2\sinh \kappa \\ \\ e^{-(q_+-q+\kappa)} - e^{-(q'-q-\kappa)} & I & e^{-(q'-q-\kappa)} - e^{-(q_+-q'+\kappa)}-2\sinh\kappa \ea \right) \]
and solve
\[  \left( \ba{ccc} e^{-(q-q'_-+\kappa)}S^{-1} - e^{-(q'-q-\kappa)} & 0 & e^{-(q'-q-\kappa)}-e^{-(q-q_-'+\kappa)} -2\sinh \kappa \\ \\ e^{-(q_+-q+\kappa)} - e^{-(q'-q-\kappa)} & I & e^{-(q'-q-\kappa)} - e^{-(q_+-q'+\kappa)}-2\sinh\kappa \ea \right) \left(\ba{c} \phi \\ \psi \\ \delta \kappa \ea \right) = \left( \ba{c} f \\ g \ea \right) \]

The second row gives $\psi$ as a linear combination of $g$, $\delta \kappa$ and a bounded multiplication operator acting on $\phi$ thus we restrict attention to the first row which is a first order, non-autonomous linear recurrence for $\phi$.  The coefficient $\alpha = e^{(2q - q'-q'_- + 2\kappa)}$ is well behaved just like the similar coefficient we studied when verifying (H1)-(C1).  In particular, the hypotheses of Lemma 2 are satisfied and we are in case 2.  To verify (H2) we must show that the $(2,3)$ entry of the derivative matrix is not orthogonal to the kernel of the adjoint.  At first glance it appears that one must dirty one's hands with the explicit form of the $m$-soliton solution in order to do this computation.  However, it was shown in \cite{benes:2011} that the quantity of interest is independent of time and hence it suffices to analyze the quantity in the limit $t \to \infty$ where it reduces to the computation for the vaccuum-kink pairing.  This computation is not difficult and has been checked in \cite{benes:2011}.  To check (C2) we again make use of the fact that a multi-soliton decomposes into the linear superposition of soliton solutions in $\ell^1$.
\end{proof}

\section*{Acknowledgements} This work was supported in part by NSF grants DMS-0908093 (CEW) and DMS-108788 (AH).  The first author thanks M. Noonan for giving him a primer on geometric B\"acklund transforms.  In addition both authors acknowledge useful conversations with G.N. Benes and D. Pelinovsky about the use of B\"acklund transformations in stability calculations.



\end{document}